\documentclass[12pt,a4paper,twoside]{article}

\usepackage[centertags]{amsmath}
\usepackage{amsfonts}
\usepackage{amssymb}
\usepackage{amsthm}
\usepackage{ulem}
\usepackage{epsfig}
\usepackage{subfigure}
\usepackage{color}
\usepackage{mathrsfs}
\usepackage[boldsans]{ccfonts}
\usepackage{array}
\usepackage{calc}
\usepackage{titlesec}
\usepackage{graphicx}
\usepackage{enumitem}

\newtheoremstyle{break}
  {\topsep}{\topsep}%
  {\itshape}{}%
  {\bfseries}{}%
  {\newline}{}%
\newtheoremstyle{break1}
  {\topsep}{\topsep}%
  {}{}%
  {\bfseries}{}%
  {\newline}{}%
\theoremstyle{break} \newtheorem{theorem}{Theorem}[section]
\theoremstyle{break} 
\theoremstyle{break1} \newtheorem{remark}[theorem]{Remark}
\theoremstyle{break} 
\theoremstyle{break1} \newtheorem{definition}[theorem]{Definition} 
\theoremstyle{break} \newtheorem{lemma}[theorem]{Lemma}
\theoremstyle{break} \newtheorem{corollary}[theorem]{Corollary}
\theoremstyle{break} 
\theoremstyle{break} 
\theoremstyle{break} 
\theoremstyle{break} 
\theoremstyle{break} 
\theoremstyle{break} 
\numberwithin{equation}{section}

\newcommand{\hide}[1]{}

\newcommand{\R}{{\mathbb{R}}}
\newcommand{\D}{{\mathbb{D}}}
\newcommand{\C}{{\mathbb{C}}}

\def\pcap{\mathop{{\rm pcap}}}
\def\lcap{\mathop{{\rm cap}}}

\begin{document}

\vspace*{-1cm}

\begin{center}
{\Large \bf Strong submultiplicativity of the Poincar\'e metric}
\end{center}

\medskip
\renewcommand{\thefootnote}{\arabic{footnote}}
\begin{center}
{\large Daniela Kraus and Oliver Roth}\\[1mm]
%\today
\end{center}
\footnotetext{Research supported
by a DFG grant (RO 3462/3--2).} 
\medskip

\centerline{\textit{To David Minda on the occasion of his retirement}}

\medskip

\begin{abstract}
We give a direct proof of an  important result of Solynin which says that the Poincar\'e metric  
is a strongly submultiplicative domain function.
This  result is then used to define a new capacity for compact
subsets of the complex plane $\C$, which might be called Poincar\'e capacity. If the compact set $K \subseteq \C$ is connected, then the Poincar\'e capacity of
$K$ is the same as the logarithmic capacity of $K$. In this special case, the
submultiplicativity is well--known and can be stated as an inequality for
the normalized conformal map onto the complement of $K$. Using the connection
between Poincar\'e metrics and universal covering maps  this
inequality is extended to the much wider class of universal covering maps.
\end{abstract}

\section{Introduction}

Let $\Omega$ be an open subset of the Riemann sphere $\hat{\C}$ and suppose that
$\Omega$ is \textit{hyperbolic} i.e.,
$\hat{\C}\backslash \Omega$ contains at least three pairwise different points.
Then $\Omega$ carries a unique complete conformal Riemannian metric
$\lambda_{\Omega}(z) \, |dz|$ with Gaussian curvature $-1$,
the so--called  hyperbolic metric or Poincar\'e metric of $\Omega$. 
That $\lambda_{\Omega}(z) \, |dz|$ has constant curvature $-1$ is equivalent
to the fact that in local coordinates the hyperbolic density $\lambda_{\Omega}(z)$ satisfies the nonlinear elliptic PDE
\begin{equation} \label{eq:curvaturedef}
 \Delta \log \lambda_{\Omega}(z)= {\lambda_{\Omega}(z)}^2 \, .
\end{equation}
Only in very rare cases it is
possible to give an explicit formula for the hyperbolic metric (see \cite{Ag68}).
It is therefore of interest to give good qualitative estimates for the
Poincar\'e metric, see e.g.~\cite{BP78,Hem79,Wei79} and the more recent
references \cite{BS2013,Bet08,GL01,KRS11,SV01,SV05}.

\medskip

The first aim of this paper is to give a full and direct  proof 
of the following beautiful sharp inequality due to Solynin \cite{S1999,S2010} which relates
the Poincar\'e metrics of two hyperbolic \textit{domains} with the Poincar\'e metric
of their union and their intersection. %For an open set $\Omega \subseteq \C$
%and a point $z \in \Omega$, we denote by $\Omega(z)$ the component of $\Omega$
% (i.e., the maximal connected open subset of $\Omega$) which
%contains the point~$z$.

\begin{theorem} \label{thm}
Let $\Omega_1$ and $\Omega_2$ be domains in $\hat{\C}$ such that 
$\Omega_1 \cap \Omega_2 \not=\emptyset$. Suppose that $\Omega_1 \cup \Omega_2$ is  hyperbolic. Then
\begin{equation} \label{eq:main}
\frac{\lambda_{\Omega_1}(z) \cdot \lambda_{\Omega_2}(z)}{\lambda_{\Omega_1
  \cup \Omega_2}(z) \cdot \lambda_{\Omega_1 \cap \Omega_2}(z)} \ge 1 
\quad \text{ for all } z \in \Omega_1 \cap \Omega_2 \, .
\end{equation}
If equality holds for one point  $z \in\Omega_1 \cap \Omega_2$, then
 $\Omega_1\subseteq \Omega_2$ or $\Omega_2 \subseteq \Omega_1$. In this case, equality
 holds for all points in $\Omega_1 \cap \Omega_2$.
\end{theorem}

\begin{remark}
It might be worth mentioning that in general it makes no sense to speak of the
value of the density of a conformal metric at a specific point on the
Riemann sphere, but it makes sense to speak of the value of the
\textit{quotient} of two such metrics,
see \cite[p.~57]{Min1982}, so the left hand side of inequality (\ref{eq:main}) as the product of two such
quotients is  indeed meaningful.
\end{remark}

\begin{remark}
As noted above, Theorem \ref{thm} is due to Solynin 
\cite{S1999,S2010} who even obtained an
extended form of inequality (\ref{eq:main}) for finitely many domains instead of just two domains, see
\cite[Theorem 2]{S1999}. In fact, Solynin obtained the estimate
(\ref{eq:main}) as a corollary to a much more general comparison result for
solutions for a certain class of nonlinear
elliptic PDEs. The proof we give below is somewhat similar, but more direct
and the details are different as we only use the classical maximum principle for subharmonic
functions. In order to be able to make use of the classical maximum principle we
first establish in Lemma \ref{lem:weak} a preliminary \textit{nonsharp} form
of the inequality (\ref{eq:main}) which can be proved analogously to the
classical Ahlfors lemma \cite{Ahl38}. What makes the proof of Lemma \ref{lem:weak} work is the fact that it is very useful to
multiply conformal metrics. This is just one of the many beautiful insights of
David Minda about conformal metrics which we have tried to learn from
him.  In order to treat the cases of equality in (\ref{eq:main}) we proceed along the lines of
David's  paper \cite{Min1987} which is concerned with the case of equality in Ahlfors' lemma.
\end{remark}

As pointed out by Solynin \cite{S1999,S2010},
Theorem \ref{thm} can be used to define for (almost) all compact subsets $K$ of a hyperbolic domain $\Omega \subseteq \hat{\C}$
a ``capacity'' in terms of the Poincar\'e metrics $\lambda_{\Omega \backslash
  K}(z) \, |dz|$ and $\lambda_{\Omega}(z)\, |dz|$ which depends on
the ``ambient'' domain $\Omega$, see Remark \ref{rem:sol1} below for details.
In the following, we propose a different definition which assigns to 
each compact subset of the complex plane ${\C}$ a capacity.
We slightly abuse notation and denote 
%for a compact subset $K$ of $\C$ which contains at least three pairwise
%different points by $\lambda_{\C \backslash K}(z) \, |dz|$ the Poincar\'e
%metric of the unbounded component of $\C \backslash K$ and 
by $ \lambda_{\hat{\C}}(z) \, |dz|$
the \textit{spherical} metric on the Riemann sphere $\hat{\C}$, that is, the unique
conformal metric on $\hat{\C}$ with constant curvature $+1$.

\begin{definition} \label{def}
Let $K$ be a compact subset of $\C$. If $K$ contains at least three pairwise different
points, then we set
$$ \pcap(K):=\frac{\lambda_{\hat{\C} \backslash
    K}(\infty)}{\lambda_{\hat{\C}}(\infty)} \, .$$
If $K$ contains at most two different points, we set $\pcap(K):=0$.
We call the number $\pcap(K)$ the \textit{Poincar\'e capacity} of the compact set $K$.
\end{definition}

Note that the definition of $\pcap(K)$ is meaningful since the quotient $\lambda_{\hat{\C} \backslash
  K}/\lambda_{\hat{\C}}$ has a well--defined value at the interior point
$\infty$ of $\hat{\C}\backslash K$ for every compact set $K \subseteq \C$ with
at least three pairwise different points.

\begin{remark} \label{rem:sol1}
As noted above, Solynin's approach  \cite{S1999,S2010} 
of  relating the Poincar\'e metric of a domain $\Omega\subseteq \hat{\C}$
with  a ``capacity'' of the complement of $\Omega$ is different from
Definition \ref{def}: Solynin fixes a hyperbolic domain $\Omega \subseteq \hat{\C}$ and a point $z_0 \in \Omega$,
and considers (the logarithm of)
$$ C_{\Omega,z_0}(K):=\frac{\lambda_{\Omega\backslash K}(z_0)}{\lambda_{\Omega}(z_0)}$$
for compact subsets $K$ of $\Omega \backslash \{z_0\}$.
Note that $C_{\Omega,z_0}(K)$ depends not only on $K$, but also on the
``ambient'' domain $\Omega$
and the point $z_0$ whereas $\pcap (K)$ only depends on $K$.
One of the main advantages of the definition of $\pcap$ is the fact that it
relates directly to universal covering maps in the same way as logarithmic
capacity relates to conformal maps (see Remark \ref{rem:1} below).
 \end{remark}

Theorem \ref{thm} immediately implies the following result.

\begin{corollary}
Let $K_1$ and $K_2$ be compact subsets of $\C$. Then
$$ \pcap (K_1 \cup K_2) \cdot \pcap (K_1 \cap K_2) \le \pcap (K_1) \cdot \pcap
(K_2) \, .$$ 
\end{corollary}

In particular, we  have the \textit{strong subadditivity
  property}
$$ \log  \pcap (K_1 \cup K_2) \le \log \pcap(K_1)+\log \pcap(K_2)-\log \pcap
(K_1 \cap K_2)$$
in the sense of Choquet's general theory of capacities \cite{Choquet} provided
that we interpret this inequality with care in the case when $K_1 \cap K_2$ contains
at most two distinct points. Therefore,
the  inequality of Theorem \ref{thm}, which in local coordinates takes the
form
$$\lambda_{\Omega_1}(z) \cdot \lambda_{\Omega_2}(z) \ge \lambda_{\Omega_1
  \cup \Omega_2}(z) \cdot \lambda_{\Omega_1 \cap \Omega_2}(z) \, ,
$$ 
 can be viewed as a \textit{strong
  submultiplicativity property} of the Poincar\'e metric.
We note that the analogous result for the \textit{logarithmic capacity}
$\lcap (K)$ of a compact set $K$ (see \cite[Chapter 5]{Ransford}), i.e.,
\begin{equation} \label{eq:cap}
 \log \lcap (K_1\cup K_2) \le \log \lcap(K_1)+\log \lcap(K_2)-\log \lcap
(K_1 \cap K_2)
\end{equation}
is  a well--known inequality in potential theory (see e.g.~\cite{Renggli} for a proof).

\begin{remark}[Poincar\'e capacity by way of universal covering maps] \label{rem:1}
It is instructive to point out an alternate way of defining
the Poincar\'e capacity  based on universal covering maps.
Let $\D:=\{z \in \C \, : \, |z|<1\}$ denote the open unit disk, let $K
\subseteq \C$ be a compact set with at least three pairwise different points,
 and denote by
 $\pi_K$  the universal covering map from $\hat{\C} \backslash
 \overline{\D}$ onto $\hat{\C} \backslash K$  normalized in such a way that $\pi_K(\infty)=\infty$ and 
$$\pi_K'(\infty):=\lim \limits_{z \to \infty} \frac{\pi_K(z)}{z}>0 \, .$$
Then $\pi_K$ has an expansion at $\infty$ of the form
\begin{equation} \label{eq:expansion}
\pi_K(z)=\pi'_K(\infty) z+\sum \limits_{k=0}^{\infty} a_j z^{-j} \, , \qquad z
\in \C \backslash \overline{\D} \, .
\end{equation}
 
Recall  that $\lambda_{\hat{\C}\backslash K}(z) \, |dz|$
and $\lambda_{\hat{\C}\backslash \overline{\D}}(z) \, |dz|$ are related via
$\pi_K$ as follows 
$$\lambda_{\hat{\C} \backslash K}(\pi_K(z)) \, |\pi_K'(z)|=\lambda_{\hat{\C}
  \backslash \overline{\D}}(z) \,  ,
\qquad z\in \hat{\C} \backslash \overline{\D} \, . $$
Since in local coordinates
$$ \lambda_{\hat{\C}}(z)=\frac{2}{1+|z|^2} \quad \text{ and } \quad 
\lambda_{\hat{\C} \backslash \overline{\D}}(z)=\frac{2}{|z|^2-1} \,
, $$  we therefore get that
\begin{eqnarray*}
 \pcap(K)&=& \frac{\lambda_{\hat{\C} \backslash K}(\infty)}{\lambda_{\hat{\C}}(\infty)}
=\lim \limits_{z \to \infty} \frac{\lambda_{\hat{\C} \backslash K}(\pi_K(z))
  \, |\pi'_K(z)|}{\lambda_{\hat{\C}}(\pi_K(z)) \, |\pi'_K(z)|}=
\lim \limits_{z \to \infty} \frac{\lambda_{\hat{\C} \backslash
    \overline{\D}}(z)}{\lambda_{\hat{\C}}(\pi_K(z)) \, |\pi'_K(z)|}\\
&=& \lim \limits_{z \to \infty} \frac{1}{|z|^2-1}\cdot 
\frac{1+|\pi_K(z)|^2}{|\pi'_K(z)|}\overset{(\ref{eq:expansion})}{=}\pi'_K(\infty) \, .
\end{eqnarray*}
\end{remark}

In view of this remark, Theorem \ref{thm} can be rephrased in the following way.

\begin{corollary}\label{cor:2}
Let $K_1$ and $K_2$ be compact subsets of $\C$ such that $K_1 \cap K_2$ contains at
least three pairwise different points. Then
\begin{equation} \label{eq:cover}
 \log \pi'_{K_1 \cup K_2}(\infty)\le \log \pi'_{K_1}(\infty)+\log
\pi'_{K_2}(\infty)-\log \pi_{K_1 \cap K_2}'(\infty) \, . 
\end{equation}
\end{corollary}

In the special case when $K$ is connected, its complement
$\hat{\C} \backslash K$ is a simply connected domain. Hence
the universal covering map $\pi_K$ is  a conformal map 
from $\hat{\C} \backslash \overline{\D}$ onto $\hat{\C} \backslash K$.
In this case the logarithmic capacity $\lcap(K)$ can be
computed as
$\lcap(K)=\pi'_K(\infty)$, see
e.g.~\cite[Corollary 9.9]{Pom}. Combined with Remark \ref{rem:1} this leads to
the following result. 

\begin{corollary}[Poincar\'e capacity vs.~logarithmic capacity] 
Let $K$ be a compact and connected subset of $\C$. Then
$\pcap(K)=\lcap(K)$.
\end{corollary}

In particular, if $K_1$, $K_2$ and $K_1\cap K_2$ are connected compact sets in
$\C$, then  $\pi_{K_1}$, $\pi_{K_2}$, $\pi_{K_1 \cap K_2}$ and $\pi_{K_1 \cup
  K_2}$ are all conformal maps. In this case,  
(\ref{eq:cover}) is a well--known inequality for conformal maps, which 
follows for instance immediately from (\ref{eq:cap}). Therefore, Corollary \ref{cor:2}
shows that the same inequality is in fact true for the much wider class of
universal covering maps. Finally, we note that if $\hat{\C}\backslash K$ is
not simply connected, then in general $\pcap(K)\not=\lcap(K)$. For instance,
take a finite set $K$ with at least three pairwise different points. Then
$\lcap(K)=0$, but $\pcap(K)>0$.

\section{Proof of Theorem \ref{thm}}

We will prove a slight extension of Theorem \ref{thm} by allowing  open sets
instead of domains. It is clearly sufficient to restrict consideration to
subsets of $\C$.
For this purpose, we
denote for an open set $U \subseteq \C$ and a point $z \in U$  by $U(z)$ the
connected component (i.e., the maximal open connected subset) of $U$ which contains the
point $z$. %If $\C\backslash U$ contains at least two different points, then
%we call the open set $U$ hyperbolic and define the Poincar\'e metric
%$\lambda_{U}(z) \, |dz|$ by $\lambda_U(z):=\lambda_{U(z)}(z)$ for every $z \in U$.

\begin{theorem} \label{thm:1}
Let $\Omega_1$ and $\Omega_2$ be open sets in $\C$ such that 
$\Omega_1 \cap \Omega_2 \not=\emptyset$ and $\Omega_1 \cup \Omega_2$ is  hyperbolic. Then
\begin{equation} \label{eq:estimate}
 \lambda_{\Omega_1}(z) \cdot \lambda_{\Omega_2}(z) \ge \lambda_{\Omega_1
  \cup \Omega_2}(z) \cdot \lambda_{\Omega_1 \cap \Omega_2}(z) 
\quad \text{ for all } z \in \Omega_1 \cap \Omega_2 \, .
\end{equation}
If equality holds for one point  $z \in\Omega_1 \cap \Omega_2$, then
 $\Omega_1(z)\subseteq \Omega_2(z)$ or $\Omega_2(z) \subseteq \Omega_1(z)$. In
 this case, equality
 holds for all points in $\Omega_1(z) \cap \Omega_2(z)$.
\end{theorem}

The proof of Theorem \ref{thm:1} will occupy the rest of this paper.
We first recall the well--known \textit{monotonicity property of the hyperbolic metric}.

\begin{lemma} \label{lem:hypmetricprinciple}
Let $\Omega\subseteq \C$ be a hyperbolic open set and let $U$ 
be an open subset of $\Omega$.  Then
\begin{equation} \label{eq:est1}
\lambda_{\Omega}(z) \le \lambda_{U}(z)
\qquad \text{ for every } z \in U \, .
\end{equation}
In particular,
\begin{equation} \label{eq:bdd1}
 \lim \limits_{z \to \xi} \lambda_{\Omega}(z)=+\infty \quad \text{ for every } \xi
\in \partial \Omega\, .
\end{equation}
%If equality holds for some point $z \in U$, then $U(z)=\Omega(z)$, so
%$\lambda_{\Omega} (z')=\lambda_{U}(z')$ for every point $z' \in U(z)=\Omega(z)$.
\end{lemma}

\begin{proof}
The estimate (\ref{eq:est1}) follows directly from the definition of the hyperbolic metric as the
\textit{maximal} conformal metric with curvature $-1$. 
In order to prove (\ref{eq:bdd1}) let $\xi,\eta$ be two different boundary
points of $\Omega$, so $\Omega \subseteq \C \backslash\{\xi,\eta\}$.
 Then
$$ \lim \limits_{z \to \xi} \lambda_{\C\backslash \{\xi,\eta\}}(z)=+\infty \,
,$$
see e.g.~\cite[formula (4.1)]{Hem79}. Since $\lambda_{\Omega}(z) \ge
\lambda_{\C\backslash \{\xi,\eta\}}(z)$ by (\ref{eq:est1}), we deduce
$\lambda_{\Omega}(z) \to +\infty$ as $z \to \xi$.
\end{proof}

In order to prove Theorem \ref{thm:1}, we need more precise information
about the boundary behaviour of $\lambda_{\Omega}$ than  provided
by Lemma \ref{lem:hypmetricprinciple}. At least for smooth open sets
such information is available
with the help of a boundary version of Ahlfors' lemma \cite{Ahl38} proved in \cite{KRR06}.

\medskip

We first make precise what we mean by ``smooth'' open sets.
We call a Jordan domain $G$ 
(i.e., a domain bounded by 
a Jordan curve in $\C$) smooth, if there is a conformal map $\phi$ from $\D$ onto
$G$ such that $|\phi'|$ extends continuously  to
$\partial \D$ with $|\phi'| \not=0$ on $\partial \D$.
By Carath\'eodory's extension theorem, this conformal map $\phi$ extends to
a homeomorphism of the closures $\overline{\D}$ and $\overline{G}$.

\begin{definition}
Let   $\Omega \subseteq \C$ 
be an open set. A subset $\Gamma$ of the boundary of $\Omega$ is called smooth 
if  for every point $\xi \in \Gamma$
there exists a smooth Jordan domain $G \subseteq \Omega$ and an open 
 neighborhood $V \subseteq \C$ of $\xi$ such that $\xi \in \partial G \cap V
 \subseteq \Gamma$. We call the open set $\Omega$ smooth, if $\partial \Omega$
 is smooth.
\end{definition}

Note that if an open set $\Omega\subseteq \C$ is bounded by finitely many
analytic Jordan
curves, then $\Omega$ is smooth.

\begin{lemma} \label{lem:ahl}
Let $\Omega \subseteq \C$ be a hyperbolic open set and let $U$ be an
open subset of $\Omega$. Suppose that $\Gamma \subseteq \partial U
\cap \partial \Omega$ is a smooth subset of the boundary of $U$ as well as of
the boundary of $\Omega$. Then
$$ \lim \limits_{z \to \xi} \frac{\lambda_{\Omega}(z)}{\lambda_U(z)} =1 \quad
\text{ for every } \xi \in \Gamma \, . $$
\end{lemma}

\begin{proof}
In view of Lemma \ref{lem:hypmetricprinciple} we have $\lambda_{\Omega}(z)  \le
\lambda_{U}(z)$  for all $z \in U$, so we only need to prove
\begin{equation} \label{eq:ahl1}
\liminf \limits_{z \to \xi} \frac{\lambda_{\Omega}(z)}{\lambda_U(z)} \ge 1
\text{ for every } \xi \in \Gamma  \, .
\end{equation}
Since $\lim_{z \to \xi} \lambda_{\Omega}(z)=+\infty$ for every $\xi \in \Gamma
\subseteq \partial \Omega$ by Lemma
\ref{lem:hypmetricprinciple} and since both metrics $\lambda_{\Omega}(z) \, |dz|$
and $\lambda_U(z) \, |dz|$ have constant curvature $-1$, we can apply
the boundary Ahlfors lemma (Theorem
5.1 in \cite{KRR06}), and this  gives (\ref{eq:ahl1}).
\end{proof}

It is always possible to exhaust a given hyperbolic set by smooth
hyperbolic sets. This is a consequence of  the following well--known
result, see e.g.~\cite[p.~91]{Str}.

\begin{lemma} \label{lem:3}
Let $\Omega_1$ and $\Omega_2$ be open sets in $\C$ such that $\Omega_1 \cap
\Omega_2\not=\emptyset$ and $\Omega_1 \cup \Omega_1$ is hyperbolic. Then for
each $n=1,2, \ldots$ there exist smooth open subsets $\Omega_{1,n}$ of $\Omega_1$ and
$\Omega_{2,n}$ of $\Omega_2$  such that
\begin{itemize}
\item[(a)] $\Omega_{1,n}$ is compactly contained in $\Omega_{1,n+1}$ and
  $\Omega_{2,n}$ is compactly contained in $\Omega_{2,n+1}$ for
  each $n=1,2,\ldots$;
\item[(b)] $\bigcup \limits_{n=1}^{\infty} \Omega_{1,n}=\Omega_1$ and
  $\bigcup \limits_{n=1}^{\infty} \Omega_{2,n}=\Omega_2$; and
\item[(c)] $\Omega_{1,n} \cap \Omega_{2,n}\not=\emptyset$   for every $n=1,2, \ldots$.
\end{itemize}
It is even possible to choose the open sets $\Omega_{1,n}$ and $\Omega_{2,n}$
in such a way that they are bounded by finitely many analytic Jordan curves.
\end{lemma}

\begin{lemma} \label{lem:4}
Let $\Omega \subseteq \C$ be a hyperbolic open set and for each $n=1,2,\ldots$
let $\Omega_{n}\not=\emptyset$ be an open subset of $\Omega$ such that
$\Omega_n \subseteq \Omega_{n+1}$ for $n=1,2 \ldots$ and $\cup_{n=1}^{\infty} \Omega_n=\Omega$.
Then for each $z \in \Omega$, we have
$$ \lim \limits_{n \to \infty}\lambda_{\Omega_n}(z)=\lambda_{\Omega}(z) \,.$$
\end{lemma}

\begin{proof}
The monotonicity property of the hyperbolic metric shows that $\lambda_{\Omega_n}(z)
\ge \lambda_{\Omega_{n+1}}(z) \ge \lambda_{\Omega}(z)$ for all $z \in \Omega$
and all (but finitely many)  positive integers $n$. Hence the limit
$$ \lambda(z):=\lim \limits_{n \to \infty} \lambda_n(z) $$
exists for every $z \in \Omega$ and $\lambda(z) \ge \lambda_{\Omega}(z)$ for
any $z \in \Omega$. By a result of Heins \cite[Lemma 11.1]{Hei62},
$ \lambda(z) \, |dz|$ is a conformal metric on $\Omega$ with constant
curvature $-1$. Since $\lambda_{\Omega}(z)\, |dz|$ is the maximal metric with
these properties, we deduce $\lambda(z)\,|dz|=\lambda_{\Omega}(z)\,|dz|$.
\end{proof}

Lemma \ref{lem:4} can also be deduced from the results of \cite{Hej74}.
We next prove Theorem \ref{thm:1} in a weak form.

\begin{lemma} \label{lem:weak}
Let $\Omega_1$ and $\Omega_2$ be open sets in $\C$ such that 
$\Omega_1 \cap \Omega_2 \not=\emptyset$ and $\Omega_1 \cup \Omega_2$ is  hyperbolic. Then
\begin{equation} \label{eq:estimateweak}
 \lambda_{\Omega_1}(z) \cdot \lambda_{\Omega_2}(z) \ge \frac{1}{\sqrt{2}}
 \cdot \lambda_{\Omega_1
  \cup \Omega_2}(z) \cdot \lambda_{\Omega_1 \cap \Omega_2}(z) 
\quad \text{ for all } z \in \Omega_1 \cap \Omega_2 \, .
\end{equation}
\end{lemma}

\begin{proof}
Let $U_1$ and $U_2$ be open sets in $\C$ which are compactly contained
in $\Omega_1$ and $\Omega_2$ respectively such that $U_1 \cap U_2\not=\emptyset$. Consider the auxiliary metric
$$ \lambda(z) \, |dz|:=\frac{\lambda_{U_1}(z) \cdot
  \lambda_{U_2}(z)}{\lambda_{U_1 \cup U_2}(z)} \, |dz|$$
on $U_1 \cap U_2$. The curvature 
$$\kappa_{\lambda}(z)=-\frac{\Delta \log \lambda(z)}{\lambda(z)^2}$$
of this metric is
$$ \kappa_{\lambda}(z)=-\left(\frac{\lambda_{U_1 \cup
    U_2}(z)}{\lambda_{U_1}(z)}\right)^2-\left(\frac{\lambda_{U_1 \cup
    U_2}(z)}{\lambda_{U_2}(z)}\right)^2+ \left(\frac{\lambda_{U_1 \cup U_2}(z)^2}{\lambda_{U_1}(z) \cdot
  \lambda_{U_2}(z)}\right)^2 \, .  $$
Since
\begin{equation} \label{eq:haha}
\lambda_{U_1 \cup U_2}(z) \le \lambda_{U_j}(z) \, , \qquad j=1,2, \, 
\end{equation}
by Lemma \ref{lem:hypmetricprinciple}, we deduce $\kappa_{\lambda}(z) \ge -2$
for all $z \in U_1 \cap U_2$. Using again (\ref{eq:haha}) and Lemma
\ref{lem:hypmetricprinciple}, we also see that
$$ \lim \limits_{z \to \xi} \lambda(z)=+\infty \quad \text{ for every }
\xi \in \partial (U_1 \cap U_2) \, .$$
Since $U_1 \cap U_2$ is compactly contained in $\Omega_1 \cap \Omega_2$, the
 continuous function $\lambda_{\Omega_1 \cap \Omega_2} : \Omega_1 \cap
\Omega_2 \to \R$ is bounded on $U_1 \cap U_2$, so
 the function
$$ u(z):=\log \left(\frac{\lambda_{\Omega_1 \cap \Omega_2}(z)}{\sqrt{2} \cdot \lambda(z)} \right) \, ,
\qquad z \in U_1 \cap U_2 \, ,$$
has a continuous extension to $\overline{U_1 \cap U_2}$  which vanishes on
$ \partial (U_1 \cap U_2)$. Now we assume that $u(z) >0$ for some $z \in U_1
\cap U_2$. Then $u$ attains its maximal value at some
point $z_0 \in U_1 \cap U_2$. This implies
\begin{eqnarray*}
0 & \ge & \Delta u(z_0)=\Delta \log \lambda_{\Omega_1 \cap \Omega_2}
(z_0)-\Delta \log \lambda(z_0)\ge  \lambda_{\Omega_1 \cap \Omega_2}(z_0)^2-2\,
\lambda(z_0)^2  \, ,
\end{eqnarray*}
so $\lambda_{\Omega_1 \cap \Omega_2}(z_0)/(\sqrt{2} \cdot \lambda(z_0)) \le
1$, that is, $u(z_0)\le 0$,
a contradiction. It follows that $u(z) \le 0$ for all $z \in U_1 \cap
U_2$, i.e., 
$$ \lambda_{U_1}(z) \cdot \lambda_{U_2}(z) \ge \frac{1}{\sqrt{2}} \cdot \lambda_{U_1
  \cup U_2}(z) \cdot \lambda_{\Omega_1 \cap \Omega_2}(z) \quad \text{ for all }
  z \in U_1 \cap U_2 \, .$$ 
This inequality holds for all open sets $U_1$ and $U_2$ which are compactly
contained in $\Omega_1$ resp.~$\Omega_2$. An application of Lemma \ref{lem:3} and Lemma
\ref{lem:4} therefore completes the proof of Lemma \ref{lem:weak}.
\end{proof}

We are now in a position to prove the inequality  of Theorem \ref{thm:1}
for the case  that $\Omega_1$ and $\Omega_2$ are bounded by finitey many
analytic arcs.

\begin{lemma} \label{lem:2}
Let $\Omega_1$ and $\Omega_2$ be   open sets in $\C$ such that 
$\Omega_1 \cap \Omega_2 \not=\emptyset$ and $\Omega_1 \cup \Omega_2$ is
hyperbolic.  Suppose that $\partial \Omega_1$ and $\partial \Omega_2$ consists
of finitely many analytic Jordan curves.
Then
$$ \lambda_{\Omega_1}(z) \cdot \lambda_{\Omega_2}(z) \ge \lambda_{\Omega_1
  \cup \Omega_2}(z) \cdot \lambda_{\Omega_1 \cap \Omega_2}(z) 
\quad \text{ for all } z \in \Omega_1 \cap \Omega_2 \, .$$
\end{lemma}

\begin{proof}
We consider the  auxiliary function
$$ u(z):=\log^+ \left( \frac{\lambda_{\Omega_1 \cap \Omega_2}(z)
  \cdot \lambda_{\Omega_1\cup\Omega_2}(z)}{\lambda_{\Omega_1}(z)
  \cdot \lambda_{\Omega_2}(z)} \right) \, , \qquad  z \in \Omega_1 \cap
\Omega_2 \, .$$
Here, $\log^+ x:=\max\{0,\log x\}$ for every positive real number $x$, so by
definition, 
$u$ is non--negative. We need to show $u(z)\equiv 0$. 

\medskip

(i) \, 
We first prove that $u$ is subharmonic
on $\Omega_1 \cap \Omega_2$. For this purpose we  note that the monotonicity principle of the hyperbolic metric (Lemma
\ref{lem:hypmetricprinciple}) shows that
\begin{equation} \label{eq:1a}
\lambda_{\Omega_1} (z) \ge \lambda_{\Omega_1 \cup \Omega_2}(z) \, , \qquad
\lambda_{\Omega_2}(z) \ge \lambda_{\Omega_1 \cup \Omega_2}(z) \text{ for all }
z \in \Omega_1 \cap \Omega_2 \, .
\end{equation}
Now fix a point $z_0 \in \Omega_1 \cap \Omega_2$ such that $u(z_0)>0$. Then we can
compute $\Delta u(z_0)$ using the curvature equation (\ref{eq:curvaturedef}). This gives us 
\begin{equation} \label{eq:2a}
\begin{array}{rcl}
 \Delta u(z_0)&=& \lambda_{\Omega_1 \cap \Omega_2}(z_0)^2
 + \lambda_{\Omega_1\cup\Omega_2}(z_0)^2 -\left( 
\lambda_{\Omega_1}(z_0)^2+\lambda_{\Omega_2}(z_0)^2 \right)
\\[2mm]
&=& \big(  \lambda_{\Omega_1 \cap \Omega_2}(z_0) -  \lambda_{\Omega_1\cup\Omega_2}(z_0) \big)^2- \big(\lambda_{\Omega_1}(z_0)-\lambda_{\Omega_2}(z_0)
  \big)^2
\\[2mm] && -2 \big( \lambda_{\Omega_1}(z_0)\lambda_{\Omega_2}(z_0)
-\lambda_{\Omega_1 \cap \Omega_2}(z_0)   \lambda_{\Omega_1\cup\Omega_2}(z_0)
\big) \, .
\end{array}
\end{equation}
Now observe that $u(z_0)>0$ is the same as $\lambda_{\Omega_1}(z_0)\lambda_{\Omega_2}(z_0)
-\lambda_{\Omega_1 \cap \Omega_2}(z_0)
\lambda_{\Omega_1\cup\Omega_2}(z_0)<0$, so
$$  \Delta u(z_0)
> \big(  \lambda_{\Omega_1 \cap \Omega_2}(z_0) -  \lambda_{\Omega_1\cup\Omega_2}(z_0) \big)^2- \big(\lambda_{\Omega_1}(z_0)-\lambda_{\Omega_2}(z_0)
  \big)^2 \, .$$
We claim that $\Delta u(z_0) \ge 0$. In fact, if $ \Delta u(z_0) < 0$, then
$$  \big( \lambda_{\Omega_1}(z_0)-\lambda_{\Omega_2}(z_0) \big)^2>  \big(
  \lambda_{\Omega_1 \cap \Omega_2}(z_0) -  \lambda_{\Omega_1\cup\Omega_2}(z_0)
\big)^2 \, .$$
Since $\lambda_{\Omega_1\cap \Omega_2}(z_0) \ge \lambda_{\Omega_1 \cup
  \Omega_2}(z_0)$ by Lemma \ref{lem:hypmetricprinciple} and since we may assume  $\lambda_{\Omega_1}(z_0) \ge
\lambda_{\Omega_2}(z_0)$,  we get 
$$ \lambda_{\Omega_1}(z_0)-\lambda_{\Omega_2}(z_0) > \lambda_{\Omega_1 \cap
  \Omega_2}(z_0) -  \lambda_{\Omega_1\cup\Omega_2}(z_0) \, .$$
This however  contradicts the montonicity property of the hyperbolic metric, which in
particular says that $\lambda_{\Omega_1}(z_0) \le \lambda_{\Omega_1 \cap
  \Omega_2}(z_0)$ and $\lambda_{\Omega_1 \cup \Omega_2}(z_0) \le \lambda_{\Omega_2}(z_0)$.
We have therefore shown that $\Delta u(z_0) \ge 0$ for every $z_0 \in \Omega_1\cap
\Omega_2$ such that $u(z_0)>0$. But since $u$ is non--negative by definition, this
easily implies that $u$ is subharmonic on $\Omega_1 \cap \Omega_2$.

\medskip

(ii) \, We now examine the boundary behaviour of the auxiliary function $u$
and  fix a point $\xi \in \partial (\Omega_1 \cap \Omega_2)$. Then $\xi
\in \partial \Omega_1 \cup \partial \Omega_2$. Let us  consider the case $\xi
\in\partial \Omega_1 \backslash \partial \Omega_2$. Then we can apply
 Lemma \ref{lem:ahl}, and this gives 
\begin{equation} \label{eq:2}
 \lim \limits_{z \to \xi}
\frac{\lambda_{\Omega_1}(z)}{\lambda_{\Omega_1 \cap \Omega_2}(z)}=1 \, .
\end{equation}
Hence, using the second inequality of (\ref{eq:1a}), we get that
$$ \liminf \limits_{z \to \xi}
\frac{\lambda_{\Omega_1}(z) \cdot \lambda_{\Omega_2}(z)}{\lambda_{\Omega_1 \cap
    \Omega_2}(z) \cdot \lambda_{\Omega_1 \cup \Omega_2}(z)} \ge 1 \, .$$
We have proved this inequality if $\xi \in \partial (\Omega_1 \cap \Omega_2)$
belongs to $\partial \Omega_1\backslash \partial \Omega_2$. Switching the roles of $\Omega_1$ and
$\Omega_2$, we see that this estimate also holds if $\xi \in \partial
\Omega_2\backslash \partial \Omega_1$, i.e.~it holds for every $\xi
\in \partial (\Omega_1 \cap \Omega_2)$ such that $\xi \not\in \partial \Omega_1
\cap \partial \Omega_2$.
This means that the auxiliary function $u$
has the property that
$$ \limsup \limits_{z \to \xi} u(z) = 0 \text{ for every } \xi \in \partial
(\Omega_1 \cap \Omega_2)\backslash (\partial \Omega_1 \cap \partial \Omega_2) \, .$$
Since  analytic Jordan curves can only
intersect at finitely many points, we deduce that
$$ %\begin{equation} %\label{eq:2}
\limsup \limits_{z \to \xi} u(z) = 0 \quad \text{ for all but finitely many } \xi \in \partial
(\Omega_1 \cap \Omega_2) \, .
$$ %\end{equation}

(iii) \, Finally we note that the auxiliary function $u$ is bounded above by
$\log \sqrt{2}$ (Lemma \ref{lem:weak}). Therefore we are in a position to apply the extended maximum
principle for subharmonic functions (see \cite[Theorem 3.6.9]{Ransford}), and this
implies that $u(z) \le 0$ for all $z \in \Omega_1 \cap \Omega_2$.
\end{proof}

\begin{proof}[Proof of Theorem \ref{thm:1}]
The estimate (\ref{eq:estimate}) of Theorem \ref{thm:1} follows immediately from Lemma \ref{lem:2},
Lemma \ref{lem:3} and Lemma \ref{lem:4}, so it remains to deal with the case
of equality. We consider the function
$$ u(z):=\log  \left( \frac{\lambda_{\Omega_1 \cap \Omega_2}(z) \cdot
    \lambda_{\Omega_1 \cup \Omega_2}(z)}{\lambda_{\Omega_1}(z) \cdot
    \lambda_{\Omega_2}(z) } \right) \, , \qquad z \in \Omega_1 \cap \Omega_2
\, .$$
Note that we have already proven that $u(z) \le 0$ in $\Omega_1 \cap
\Omega_2$. As in the proof of Lemma \ref{lem:2}, we have
\begin{equation} \label{eq:equal1}
\begin{array}{rcl}
\Delta u(z) &=& \big(
  \lambda_{\Omega_1 \cap \Omega_2}(z)-\lambda_{\Omega_1 \cup \Omega_2}(z)
\big)^2-\big( \lambda_{\Omega_1}(z)-\lambda_{\Omega_2}(z) \big)^2\\[2mm] & & \quad   +2
\big( \lambda_{\Omega_1\cap \Omega_2}(z) \lambda_{\Omega_1 \cup
    \Omega_2}(z)-\lambda_{\Omega_1}(z) \lambda_{\Omega_2}(z)\big) \, .
\end{array}
\end{equation}

Now, Lemma \ref{lem:hypmetricprinciple} implies $\lambda_{\Omega_1 \cap
  \Omega_2}(z) \ge \lambda_{\Omega_j}(z) \ge \lambda_{\Omega_1 \cup
  \Omega_2}(z)$ for $j=1,2$, so
$$ \lambda_{\Omega_1 \cap \Omega_2}(z)-\lambda_{\Omega_1 \cup \Omega_2}(z) \ge
|\lambda_{\Omega_1}(z)-\lambda_{\Omega_2}(z)| \, .$$
Inserting this into (\ref{eq:equal1}), we get
\begin{equation} \label{eq:equal2}
 \Delta u(z) \ge 2 \big( \lambda_{\Omega_1 \cap \Omega_2}(z) \lambda_{\Omega_1 \cup
    \Omega_2}(z)-\lambda_{\Omega_1}(z) \lambda_{\Omega_2}(z)\big) \, .\end{equation}
Applying the elementary inequality
$$ a \log \frac{b}{a} \le b-a \text{ for all } a,b \in \R, \, a \ge b >0 $$
for $a=\lambda_{\Omega_1}(z) \lambda_{\Omega_2}(z)$ and $b=\lambda_{\Omega_1
  \cap \Omega_2}(z) \cdot \lambda_{\Omega_1 \cup \Omega_2}(z)$, we deduce from
(\ref{eq:equal2}) that
$$ \Delta u(z) \ge 2 \lambda_{\Omega_1}(z) \lambda_{\Omega_2}(z) \cdot u(z) \,
, \qquad z \in \Omega_1 \cap \Omega_2 \, .$$
Hence, the strong maximum principle of Hopf (see \cite[Thm~2.1.2]{Jo}) implies
that in each connected component of $\Omega_1 \cap \Omega_2$
either $u=0$ or $u<0$. Therefore, if $z_0 \in \Omega_1 \cap \Omega_2$ is a
point such that equality holds in (\ref{eq:estimate}),
 then equality holds in
(\ref{eq:estimate}) for all $z$ in  $(\Omega_1 \cap \Omega_2)(z_0)$, the
component of $\Omega_1 \cap \Omega_2$ which contains $z_0$. In view of (\ref{eq:equal1})
this implies
\begin{equation} \label{eq:equal3}
 \big(
  \lambda_{\Omega_1 \cap \Omega_2}(z)-\lambda_{\Omega_1 \cup \Omega_2}(z)
\big)^2=\big( \lambda_{\Omega_1}(z)-\lambda_{\Omega_2}(z) \big)^2 \, , \qquad
z \in (\Omega_1 \cap \Omega_2)(z_0) \, .
\end{equation}
We need to show that $\Omega_1(z_0) \subseteq \Omega_2(z_0)$ or $\Omega_2(z_0)
\subseteq \Omega_1(z_0)$. This is obviously true if $(\Omega_1
\cap\Omega_2)(z_0)=(\Omega_1 \cup \Omega_2)(z_0)$, so we need to analyze the
case  $(\Omega_1 \cap \Omega_2(z_0) \subsetneq
(\Omega_1\cup \Omega_2)(z_0)$. 
Since the  hyperbolic metric is \textit{strictly}
decreasing with expanding domains (see \cite[p.~683]{HK}), we have
 $\lambda_{\Omega_1 \cap \Omega_2}> \lambda_{\Omega_1 \cup \Omega_2}$ in
 $(\Omega_1 \cap \Omega_2)(z_0)$
and therefore the identity
(\ref{eq:equal3}) and the continuity of
$\lambda_{\Omega_1}-\lambda_{\Omega_2}$  imply that there either $\lambda_{\Omega_1}>\lambda_{\Omega_2}$
or $\lambda_{\Omega_1}<\lambda_{\Omega_2}$. In the first case, we can deduce
from (\ref{eq:equal3}) that
\begin{equation} \label{eq:equal4}
\lambda_{\Omega_1 \cap \Omega_2}(z)-\lambda_{\Omega_1 \cup \Omega_2}(z)
=\lambda_{\Omega_1}(z)-\lambda_{\Omega_2}(z)  \,  , \quad z \in (\Omega_1
\cap \Omega_2)(z_0)\, .
\end{equation}
This shows  $(\Omega_1 \cap \Omega_2)(z_0)=\Omega_1(z_0)$, because otherwise
$\lambda_{\Omega_1 \cap \Omega_2}(z)>\lambda_{\Omega_1}(z)$, i.e.,
$\lambda_{\Omega_1 \cup \Omega_2}(z) =\lambda_{\Omega_1 \cap
  \Omega_2}(z)-\lambda_{\Omega_1}(z)+\lambda_{\Omega_2}(z)>\lambda_{\Omega_2}(z)$,
a contradiction. Hence we get $\lambda_{\Omega_1 \cap
\Omega_2}(z)=\lambda_{\Omega_1}(z)$ for every $z \in (\Omega_1 \cap
\Omega_2)(z_0)$. Therefore, (\ref{eq:equal4}) shows $\lambda_{\Omega_1 \cup
  \Omega_2}=\lambda_{\Omega_2}$  on $(\Omega_1 \cap \Omega_2)(z_0)$, so
$(\Omega_1 \cup \Omega_2)(z_0)=\Omega_2(z_0)$. Putting all this together gives
$\Omega_1(z_0)=(\Omega_1 \cap \Omega_2)(z_0) \subsetneq (\Omega_1 \cup \Omega_2)(z_0)=\Omega_2(z_0)$.
\end{proof}

\begin{remark}
The approximation technique by smooth open sets, which has been used to prove
Theorem \ref{thm:1}, cannot be avoided entirely. This follows from the fact
that if e.g.~$\Omega_1$ is a non--smooth open set, then 
the crucial limit relation (\ref{eq:2}), which in turn is based on Lemma \ref{lem:ahl}, is no longer valid. To see this, take
$\Omega_1:=\C \backslash \{-1,1\}$ and
$\Omega_2:=\D$, so $\Omega_1 \cap \Omega_2=\D$ and
$$ \lambda_{\Omega_1 \cap \Omega_2}(z)=\lambda_{\D}(z)=\frac{2}{1-|z|^2} \, ,$$
On the other hand, it is known (see \cite{Min}) that
$$ \lim \limits_{z \to 1} \lambda_{\C \backslash \{-1,1}(z) |z-1| \log \left(
  \frac{1}{|z-1|} \right)=1 \, .$$
Hence,
$$ \lim \limits_{z \to 1}  \frac{\lambda_{\Omega_1}(z)}{\lambda_{\Omega_1 \cap
    \Omega_2}(z)}=\lim \limits_{z \to 1} \frac{\lambda_{\C \backslash
    \{-1,1\}}(z)}{\lambda_{\D}(z)}= 0 \, .
$$
It would be interesting to see whether the regularity conditions imposed on the boundary
set $\Gamma$ in Lemma \ref{lem:ahl} can be considerably weakend.
\end{remark}

\section{Concluding remark}

We conclude this paper by noting that 
Theorem \ref{thm} and Theorem \ref{thm:1} remain valid, for  subdomains (or
open subsets) of a Riemann surface since it  
makes sense to speak of the value of the \textit{quotient} of two 
conformal metrics at a specific point on a Riemann surface,
see again \cite{Min}. Hence the ``Riemann surface version'' of Theorem \ref{thm}
takes the following form.

\begin{theorem}
Let $R$ be a Riemann surface and let  $\Omega_1,\Omega_2\subseteq R$ be domains  such that 
$\Omega_1 \cap \Omega_2 \not=\emptyset$ and $\Omega_1 \cup \Omega_2$ is  hyperbolic. Then
$$ \frac{\lambda_{\Omega_1}(z) \cdot \lambda_{\Omega_2}(z)}{\lambda_{\Omega_1
  \cup \Omega_2}(z) \cdot \lambda_{\Omega_1 \cap \Omega_2}(z) } \ge 1
\quad \text{ for all } z \in \Omega_1 \cap \Omega_2 \, .$$
If equality holds for one point  $z \in\Omega_1 \cap \Omega_2$, then
 $\Omega_1\subseteq \Omega_2$ or $\Omega_2 \subseteq \Omega_1$. In this case, equality
 holds for all points in $\Omega_1 \cap \Omega_2$.
\end{theorem}

The proof is almost identical to the proof of Theorem \ref{thm} with obvious
modifications.

\vfill
Daniela Kraus {\small (dakraus@mathematik.uni-wuerzburg.de)}\\
Oliver Roth {\small (roth@mathematik.uni-wuerzburg.de)}\\
Department of Mathematics\\
University of W\"urzburg\\
Emil Fischer Stra{\ss}e 40\\
97074 W\"urzburg\\
Germany

\end{document}